\documentclass{amsart}
\usepackage{graphicx }
\usepackage[all]{xy}
\newcommand{\Z}{\mathbb Z}
\newcommand{\R}{\mathbb R}
\newcommand{\N}{\mathbb N}

\usepackage{enumerate,comment}
\newtheorem{thm}{Theorem}[section]

\newtheorem{lem}[thm]{Lemma}
\newtheorem{prop}[thm]{Proposition}
\newtheorem{cor}[thm]{Corollary}

\newtheorem{prob}[thm]{Problem}

\newtheorem{defn}[thm]{Definition}

\newtheorem{mainthm}[thm]{Main Theorem}

\begin{document}

\title[TAKENS-TYPE RECONSTRUCTION THEOREMS]{Takens-type reconstruction theorems of one-sided dynamical systems on compact metric spaces}

\author[Kato]{Hisao Kato}

\email[Kato]{hkato@math.tsukuba.ac.jp}

\address[Kato]{Institute of Mathematics, University of Tsukuba,
Tsukuba, Ibaraki  305-8571, Japan}

\keywords{Takens' reconstruction theorem of
dynamical systems, one-sided dynamical systems, time-delay embedding, topological dimension, periodic points, chain recurrent points, branched manifolds, Menger manifolds, Sierpi\'nski carpet, Sierpi\'nski gasket,  fractal sets}
\subjclass[2010]{Primary 37C45; Secondary, 37M10, 54H20, 54F45}

\maketitle
{\bf Abstract.} The reconstruction theorem deals with dynamical systems that are given by  a map $T:X\to X$ of a compact metric space $X$ together with an observable $f:X \to \R$ from $X$ to the real line $\R$. In 1981, by use of  Whitney's embedding theorem, Takens proved that if $T:M\to M$ is a diffeomorphism on a compact smooth manifold $M$ with $\dim M=d$, for generic $(T,f)$ there is a bijection between elements $x \in M$ and corresponding sequence $(fT^j(x))_{j=0}^{2d}$, and moreover, in 2002 Takens proved a generalized version  for endomorphisms. 

In natural sciences and physical engineering, there has been an increase in importance of fractal sets and more complicated spaces, and also in mathematics, many topological and  dynamical properties and stochastic analysis of such spaces have been studied. 
In the present paper, by use of some topological methods  we extend the Takens' reconstruction theorems of compact smooth manifolds to reconstruction theorems of one-sided dynamical systems for a large class of compact metric spaces, which contains PL-manifolds, branched manifolds and some fractal sets, e.g. Menger manifolds, Sierpi\'nski  carpet and Sierpi\'nski gasket and dendrites, etc.

\section{Introduction} 
Throughout this paper, all spaces are separable metric spaces
and maps are
continuous
functions. Let $\N$ be the set of all nonnegative integers, i.e., $\N
=\{0,1,2,...\}$ and let $\Z$ be  the set of all integers and $\R$ the real line. 

A map $h:X \to Y$  is an {\it embedding} if $h:X\to h(X)$ is a homeomorphism. 
A pair $(X,T)$ is called a {\it one-sided dynamical system} (abbreviated as {\it dynamical system})  if $X$ is a separable metric space and $T:X \to X$ is any map. Moreover, if $T:X \to X$ is a
homeomorphism, i.e., invertible, then $(X,T)$ is called a {\it two-sided dynamical system}.

Reconstruction of dynamical systems from a scalar time series is a topic that has been extensively studied. The theoretical basis for methods of recovering dynamical systems on compact manifolds from one-dimensional data was studied by Takens \cite{Tak81,Tak02}.  
 In 1981, Takens \cite{Tak81}, by use of Whitney's embedding theorem, proved that under some conditions of (two-sided) diffeomorphisms on a manifold, the dynamical
system can be reconstructed from the observations made with generic functions.

\begin{thm}
{\rm  (Takens' reconstruction theorem for diffeomorphisms)}  Suppose that
$M$ is a compact smooth manifold of dimension $d$. Let
$D^r(M,M)$ be the space of all $C^r$-diffeomorphisms on $M$and $C^r(M,\R)$ the set of
all $C^r$-functions $(r \geq 1)$. If $E$ is the set of all pairs $(T,f) \in D^r(M,M)\times C^r(M,\R)$ 
such that the delay observation map 
$I_{T,f}^{(0,1,2,..,2d)}:M\to \R^{2d+1}$  defined  by 
$$x \mapsto (fT^j(x))_{j=0}^{2d}$$ 
is an embedding, then $E$ is open and dense in $D^r(M,M)\times C^r(M,\R)$.
\end{thm}

Moreover, in 2002 Takens \cite{Tak02}, extended his theorem for  
 endomorphisms on compact  smooth manifolds as follows.  

\begin{thm}{\rm  (Takens' reconstruction theorem for endomorphisms)} Suppose that
$M$ is a compact smooth manifold of dimension $d$. Then there is an open dense subset ${\mathcal U}\subset {\rm End}^1(M)\times C^1(M,\R)$, where  $ {\rm End}^1(M)$ denotes the space of all $C^1$-endomorphisms on $M$, such that, whenever $(T,f)\in {\mathcal U}$, there is a map $$\pi:I_{T,f}^{(0,1,..,2d)}(M)\to M$$ with $\pi\cdot I_{T,f}^{(0,1,..,2d)}=T^{2d}$. 

\end{thm}

Embeddings of two-sided dynamical systems in the two-sided shift $(\R^{\Z},\sigma)$ have been studied by many authors (e.g. see \cite{Coo15, Gut15, Gut16, GQS18, GT14, Jaw74, Lin99, LW00, Ner91, SYC91, Tak81}).

In \cite{Kat20}, we studied embeddings of one-sided dynamical systems  in the one-sided shift $(\R^{\N},\sigma)$.
 In this paper, by use of the topological methods introduced in the  paper \cite{Kat20}, we extend the above Takens' reconstruction theorems of dynamical systems on compact manifolds to  theorems of  one-sided dynamical systems for a large class of compact metric spaces.  The main results of this paper are Theorem 5.4 and Corollary 6.9. 
 
  In this paper, we do not assume injectivity 
of $T$ and  so the proofs of our results cannot any longer rely on the embedding theorems of Whitney and Menger-N\"{o}beling \cite{Eng95}.  Instead, an essential role is played by the  notion defined in Definition 2.1.  \\
\\

\section{Definitions and notations}

For a space $X$, $\dim X$ means the topological (covering)
dimension of $X$
(e.g. see \cite{Eng95}, \cite{HW41} and \cite{Nag65}).  
Let $X$ be compact metric space and $Y$ a space with a complete
metric $d_Y$.
Let $C(X,Y)$ denote the space consisting of all maps $f: X \to
Y$. We equip
$C(X,Y)$ with the metric $d$ defined by $$d(f,g)=\sup_{x\in
X}d_Y(f(x),g(x)).$$
Recall that $C(X,Y)$ is a complete metric space and hence
Baire's category
theorem holds in $C(X,Y)$.

A map $g:X \to Y$ of separable metric spaces is {\it $n$-dimensional}
~$(n=0,1,2,...)$ if $\dim g^{-1}(y)\leq n$ for
each $y \in Y$.  Note that a closed map $g:X \to Y$ is 0-dimensional if and only if for any 0-dimensional subset $D$ of $Y$, $\dim g^{-1}(D)\leq 0 $ (see \cite[Hurewic's theorem (1.12.4)]{Eng95}). A map $T:X\to X$ is {\it doubly 0-dimensional} if for each closed set $A\subset X$ of dimension 0, one has $\dim T^{-1}(A) \leq 0$ and $\dim T(A)=0$. 

 If
$K$ is a subset of a space $X$, then $\mathrm{cl}(K)$,
$\mathrm{bd}(K)$ and
$\mathrm{int}(K)$ denote the closure, the boundary and the
interior of
$K$ in $X$, respectively. A subset $A$ of a space $X$ is an
$F_{\sigma}$-set of
$X$ if $A$ is a countable union of closed subsets of $X$. Also,
a subset $B$ of
$X$ is a $G_{\delta}$-set of $X$ if $B$ is an intersection of
countably many  open
subsets of $X$.

An indexed family $(C_s)_{s\in S}$ of subsets of a set $X$ will by abuse  of notation also be denoted by $\{C_s\}_{s\in S}$ or $\{C_s: s\in S\}$. Hence if $\mathcal{C}=\{C_s\}_{s\in S}$ is such a family then its members $C_s$ and $C_t$ will be considered as different whenever $s\ne t$. We then put
$$\mathrm{ord}(\mathcal{C})= \sup \{\mathrm{ord}_x(\mathcal{C}): x\in X\} 
\mbox{, where } \mathrm{ord}_x (\mathcal{C})=|\{s\in S|~x \in C_s\}|.$$
Note that $\mathrm{ord}(\mathcal{C})$ so defined is by  1 larger than it would be according to the usual definition, as e.g. in \cite[(1.6.6) Definition]{Eng95}.  \\

Modifying the definition of TSP in \cite{Kat20}, we define the notion of 
$(k,\eta)$ trajectory-separation property for $k \in \N$ and $\eta >0$ which is very  important in this paper.  
 
\begin{defn}
Let $T: X \to X$ be a map of a compact metric space $X$ with
$\dim X = d < \infty$ and let $k\in \N, \eta >0$. Then $T$ has the $(k,\eta)$ {\it
trajectory-separation property} $((k,\eta)$-TSP for short$)$ provided that
there is a closed set $H$ of $X$
such that \\
$(1)$ $X\setminus  H$ is a union of finitely many disjoint open sets
 of  diameter at most $\eta$,  and \\
$(2)$ $\mathrm{ord}\{T^{-p}(H)\}_{p=0}^k \leq d$.

\end{defn} 

\section{reconstruction spaces of dynamical systems}

For a space $K$, we consider the (one-sided) shift $\sigma:
K^{\N} \to
K^{\N}$ which is defined by
$$\sigma (x_0,x_1,x_2,x_3....) = (x_1,x_2,x_3....), ~x_i \in
K.$$ 
Let $(X,T)$ and $(X',T')$ be dynamical systems.  If a map $h: X\to X'$  satisfies $hT=T'h$, 
then we say that  $h: (X,T)\to (X',T')$ is a {\it morphism} of dynamical systems.

In this paper, we need the following definition from \cite{Kat20}. 

\begin{defn} Let $T:X\to X$ be a map of a compact metric space $X$.\\
$(a)$ Given a set $S \subset \N$ and a map $f:X \to \R$, the map $(fT^j)_{j\in S}: X \to \R^S$ will be denoted by $I^S_{T,f}$. We call this map {\it the delay observation map at times} $j\in S$.  Note that $I_{T,f}:=I_{T,f}^{\N}:(X,T)\to (\R^{\N},\sigma)$ is a morphism of dynamical systems. We call $I_{T,f}$ the infinite delay observation map for $(T,f)$. \\
$(b)$ We say that $I_f^S$ is a {\it trajectory-embedding} if  $I_f^S(x)\neq I_f^S(y)$ whenever $T^j(x)\neq T^j(y)$ for all $j \in S$. 
\end{defn}

Remark 1. $(1)$ In the statement of Theorem 1.2, the existence of such a map $\pi$ is equivalent to that $I_{T,f}^{(0,1,...,2d)}$ is a trajectory-embedding. \\
$(2)$  In the statement $(b)$ of Definition 3.1, for the case where $T:X\to X$ is injective, 
$I_{f}^S$ is an embedding if and only if $I_{f}^S$ is a trajectory-embedding. \\

 Let $(X,T)$ be a dynamical system of a compact metric space $X$. 
For $n\geq 1$, let $P_n(T)$ be the set of all periodic points of $T$
with period $\leq n$
 and $P(T)$ the set of all periodic points of $T$, i.e. 
$$P_n(T)=\{x\in X|~\mbox{there is an}~i ~\mbox{such that}~1\leq i\leq
n~\mbox{and}~T^{i}(x)=x\} $$
$$\mbox{and}~ P(T)=\bigcup_{n\geq 1}P_n(T).$$
Two points $x$ and $y$ of $X$ are {\it trajectory-separated} for $T$ if $T^j(x)\neq T^j(y)$ 
for $j\in \N$. A morphism $h:(X,T)\to (X',T')$ is a {\it trajectory-monomorphism} if 
 $h(x), h(y)$ are trajectory-separated for $T'$, whenever $x, y \in X$ are 
  trajectory-separated for $T$. 
   
 For $x,y \in X$, let $o_T (x)=(T^i(x))_{i\in \N}$ and $o_T (y) = (T^i(y))_{i\in \N}$ be two orbits of $T$. We say that the orbit $o_T(x)$  is
{\it eventually equivalent} to the orbit $o_T(y)$ if the orbits will be equal in the future, i.e.,
there exists an $n \in \N$ such that $T^i(x)=T^i(y)$  for each $i\geq n$. In this case, we
wright $o_T(x) \sim_e o_T(y)$. We see that this relation is an equivalence relation. So we
have the equivalence class 
$$[o_T(x)]=\{o_T(y)|~ o_T(x) \sim_e o_T(y)\}$$
containing $o_T (x)$ and we put
$$[O(T)] = \{[o_T (x)]|~ x \in X\}.$$ Note that if $T:X \to X$ is injective, the function $o:X\to [O(T)]$ defined by $x \mapsto [o_T(x)]$ is bijective, i.e.,  $o: X \cong [O(T)]$.  
Also, note that if $h:(X,T)\to (X',T')$ is a morphism of dynamical 
systems, then $h$ induces the function $h:[O(T)] \to  [O(T')]$ 
defined by $h([o_T(x)])=[o_{T'}(h(x))]$ for $x\in X$. A morphism $h:(X,T)\to (X',T')$ of dynamical systems is a {\it trajectory-isomorphism} if $h$ induces the bijection $h:[O(T)] \cong  [O(T')]$.  

\begin{prop}
Suppose that a morphism  $h: (X,T)\to (X',T')$ is a trajectory-monomorphism and  $h$ is surjective, i.e., $h(X)=X'$.  Then  $h$ is a trajectory-isomorphism:  
$$h:[O(T)]\cong [O(T')]$$ 
\end{prop}

\begin{proof}
Since  $h$ is a trajectory-monomorphism, $h$ induces an injective  function from $[O(T)]$ to $[O(T')]$. 
Also $h$ induces a surjective function from $[O(T)]$ onto $[O(T')]$, because that 
$h$ is a surjective function.
\end{proof}

We need the definition of topological entropy and we give the definition  by Bowen \cite{Bow78}. 
Let $T:X\to X$ be any map of a compact metric space $X$. 
A subset $E$ of X is $(n,\epsilon)$-separated
if for any $x,y \in E$ with $x\neq y$, there is an integer $j$ 
 such that $0 \leq j < n$ and
$d(T^j(x),T^j(y)) \geq \epsilon$. If $K$ is any nonempty closed subset of $X$, 
$s_n(\epsilon;K)$ denotes the
largest cardinality of any set $E\subset K$ which is $(n,\epsilon)$-separated. 
Also we define 
$$s(\epsilon;K)=\limsup_{n\to \infty}\frac{1}{n}\log s_n(\epsilon;K),$$
$$h(T;K)=\lim_{\epsilon \to 0}s(\epsilon;K).$$
It is well known that the topological entropy $h(T)$ of $T$ is equal to $h(T;X)$ (see
\cite{Bow78}).

Now, we will introduce the notion of {\it reconstruction space}  of   dynamical systems. 

\begin{defn} 
A compact metric space $X$ is a reconstruction space of   dynamical systems if there exists a $G_{\delta}$-dense set $E$ of $C(X,X)\times C(X,\R)$ such that for $(T,f) \in E$,  the infinite delay observation map  $$I_{T,f}:=I_{T,f}^{\N}:(X,T) \to (\R^{\N},\sigma)$$  satisfies the following conditions (1) and (2):\\ 
$(1)$ $I_{T,f}:[O(T)] \cong [O(\sigma_{T,f})]$, where  $\sigma_{T,f}=\sigma|I_{T,f}(X)$, and \\
$(2)$ $h(T)=h(\sigma_{T,f})$.
\end{defn}

We show that many compact metric spaces (e.g. PL-manifolds, branched manifolds, Menger manifolds, Sierpi\'nski carpet, Sierpi\'nski gasket and many fractal sets) are reconstruction spaces of dynamical systems. Our result means that almost all dynamical systems $(X,T)$ on a reconstruction space $X$ can be reconstructed from (observation) maps $f:X\to \R$ in the sense of  `eventually equivalent orbits', and so it forms a bridge between the theory of nonlinear dynamical
systems and nonlinear time series analysis.\\

\section{Trajectory-embeddings in $(\R^{\N},\sigma)$}
In this section, we study some fundamental properties of trajectory-embeddings.   

\begin{prop} Let $(X,T)$ be a dynamical system and $f:X\to \R$ a map. Let  $k\in \N$ and suppose that  
$I_{T,f}^{(0,1,..,k)}:X\to \R^{k+1}$ is a trajectory-embedding. 
Then the following properties  $(1)$-$(3)$ hold.  \vspace{2mm}\\
$(1)$ 
There is the unique map  $\sigma_{T,f}^{(0,1,..,k)}: I_{T,f}^{(0,1,..,k)}(X)\to I_{T,f}^{(0,1,..,k)}(X)$ such that    
 $I_{T,f}^{(0,1,..,k)}T=\sigma_{T,f}^{(0,1,..,k)}I_{T,f}^{(0,1,..,k)}$.

In other words, the  map $\sigma_{T,f}^{(0,1,..,k)}$ defined by $(fT^i(x))_{i=0}^k) \mapsto (fT^i(x))_{i=1}^{k+1}) ~ (x \in X)$ is well-defined. 
And $I_{T,f}^{(0,1,..,k)}:(X,T)\to (I_{T,f}^{(0,1,..,k)}(X),\sigma_{T,f}^{(0,1,..,k)})$ is a  trajectory-isomorphism. In particular,  $I_{T,f}:=I_{T,f}^{\N}:(X,T) \to (R^{\N},\sigma)$ is a 
 trajectory-monomorphism. \vspace{1mm}\\
$(2)$ Let  $p_{(0,1,..,k)}:\R^{\N} \to \R^{k+1}$ be the projection defined by $(x_i)_{i\in \N} \mapsto (x_i)_{i=0}^{k}$. Then   
$p_{(0,1,..,k)}:(I_{T,f}(X),\sigma_{T,f}) \to (I_{T,f}^{(0,1,..,k)}(X),\sigma_{T,f}^{(0,1,..,k)})$ is an isomorphism of dynamical systems, i.e., $p_{(0,1,..,k)}$ is a homeomorphism.  \vspace{1mm}\\
$(3)$ $h(T)=h(\sigma_{T,f})=h(\sigma_{T,f}^{(0,1,..,k)})$.
\end{prop}

\begin{proof} We prove $(1)$. 
Let $x, y\in X$ with $(fT^i(x))_{i=0}^k=(fT^i(y))_{i=0}^k$. We show that $(fT^i(x))_{i=1}^{k+1}=(fT^i(y))_{i=1}^{k+1}$. 
 Since $I_{T,f}^{(0,1,..,k)}$ is  a trajectory-embedding, we see that $T^k(x)=T^k(y)$. In particular, $T^{k+1}(x)=T^{k+1}(y)$ and so $fT^{k+1}(x)=fT^{k+1}(y)$. This implies that $(fT^i(x))_{i=1}^{k+1}=(fT^i(y))_{i=1}^{k+1}$. Thus  $\sigma_{T,f}^{(0,1,..,k)}$ is well-defined. Also, since $I_{T,f}^{(0,1,..,k)}:X \to \R^{k+1}$ is a trajectory-embedding, we see that the morphism 
$$I_{T,f}^{(0,1,..,k)}:(X,T)\to (I_{T,f}^{(0,1,..,k)}(X),\sigma_{T,f}^{(0,1,..,k)})$$ is a trajectory-isomorphism. 

We will prove $(2)$. Note that 
$p_{(0,1,..,k)}(I_{T,f}(X))=I_{T,f}^{(0,1,..,k)}(X)$. Suppose that $I_{T,f}^{(0,1,..,k)}(x)=I_{T,f}^{(0,1,..,k)}(y)~(x, y \in X)$, i.e., 
$(fT^i(x))_{i=0}^{k}=(fT^i(y))_{i=0}^{k}$. Since $I_{T,f}^{(0,1,..,k)}$ is  a trajectory-embedding, as before we see that  $T^k(x)=T^k(y)$ and so 
$$I_{T,f}(x)=(fT^i(x))_{i \in \N}=(fT^i(y))_{i \in \N}=I_{T,f}(y).$$ This implies that $p_{(0,1,..,k)}: I_{T,f}(X)\to I_{T,f}^{(0,1,..,k)}(X)$ is a homeomorphism. 

We prove (3).  Recall that $I_{T,f}^{(0,1,..,k)}T=\sigma_{T,f}^{(0,1,..,k)}I_{T,f}^{(0,1,..,k)}$. 
By Bowen's theorem (e.g. see \cite[Theorem 7.1]{MS93}), we have 
$$h(\sigma_{T,f}^{(0,1,..,k)}))\leq h(T)\leq h(\sigma_{T,f}^{(0,1,..,k)})+\sup\{h(T;(I_{T,f}^{(0,1,..,k)})^{-1}(z))|~z\in Z\}.$$
Let $z=(fT^i(x))_{i=0}^k~(x\in X)$. Since $I_{T,f}^{(0,1,..,k)}: X \to \R^{k+1}$ is a trajectory-embedding, we see that 
$$T^{k}((I_{T,f}^{(0,1,..,k)})^{-1}(z))=\{T^k(x)\}$$ is a one point set, and  so $h(T;(I_{T,f}^{(0,1,..,k)})^{-1}(z))=0$. Hence 
$h(T)=h(\sigma_{T,f}^{(0,1,..,k)})$. By (2), $h(T)=h(\sigma_{T,f}^{(0,1,..,k)})=h(\sigma_{T,f})$. 
\end{proof}

 By Proposition 4.1 and \cite[Theorem 3.1]{Kat20}, we have the following result.

\begin{thm}
Let $X$ be a compact metric space with
$\dim X = d <
\infty$
and let $T: X \to X$ be  a doubly 0-dimensional map with $\dim
P(T) \leq 0$. Then there is a dense $G_{\delta}$-set $D$ of
$C(X,\R)$ such that for all $f \in D$, 
$$I_{T,f}=T_{T,f}^{\N}:(X,T) \to (\R^{\N},\sigma)$$  satisfies the three conditions: \\
$(a)$ $I_{T,f}:[O(T)] \cong [O(\sigma_{T,f})]$, \\
$(b)$ $h(T)=h(\sigma_{T,f})$ and \\
$(c)$ if $x, y\in X$ are trajectory-separated for $T$, then 
$$|\{i\in \N|~I_{T,f}(x)_i=I_{T,f}(y)_i\}|\leq 2d.$$

\end{thm}

\section{Reconstruction theorem in the one-sided shift $(\R^{\N},\sigma)$}

Let $X, Y$ be compact metric spaces and let $\varphi:X\to 2^Y\cup \{\emptyset \}$ be a set-valued function, where $2^Y$ denotes the set of all nonempty closed subsets
of $Y$. Then  $\varphi:X\to 2^Y\cup \{\emptyset \}$ is {\it upper semi-continuous} if 
for any $x\in X$ and any open neighborhood $V$ of $\varphi(x)$ in $Y$, there is an open neighborhood $U$ of $x$ in $X$ such that $\varphi(x') \subset V$ for any $x'\in U$. 

Let $(X,T)$ be any one-sided dynamical system.  A point $x\in X$ is a {\it chain recurrent point} of $T$ if for any $\epsilon >0$ 
there is a finite sequence $x=x_0,x_1,\cdots,x_m=x~(m\geq 1)$ of points of $X$ such that 
$d(T(x_i),x_{i+1}) < \epsilon$  for
each $i=0,1,\cdots,m-1$.  Let $CR(T)$ be the set of all chain recurrent points of $T$. Note
that $P(T) \subset CR(T)$, $CR(T)$ is a nonempty closed subset of $X$ and the 
set-valued function
$$CR : C(X,X) \to 2^X, ~~T  \mapsto  CR(T)$$
is upper semi-continuous (see \cite{BF85}). 

We will define the following class 0-$\mathcal{DCR}$ of compact metric spaces. 
\begin{defn}
Let 0-$\mathcal{DCR}$ be the class  of all compact metric spaces $X$ satisfying  the following two  conditions: \\ 
$(0$-$\mathcal{D})$ The set of  doubly 0-dimensional maps $T:X \to X$ is dense in $C(X,X)$. \\
$(0$-$\mathcal{CR})$ The set of maps $T: X \to X$ with  $\dim CR(T)=0$ is dense in $C(X,X)$. 
\end{defn} 

Remark 2. Note that for a compact metric space $X$, both the set  of  0-dimensional maps $T:X\to X$  and the set of maps $T:X \to X$ with  $\dim CR(T)=0$ are $G_{\delta}$-sets of $C(X,X)$ (e.g. see \cite{KOU16}). So note that if $X$ belongs to 0-$\mathcal{DCR}$,  then the set  of  all maps $T:X\to X$ such that $T$ is a 0-dimensional map with $\dim CR(T)=0$ is a dense $G_{\delta}$-set of $C(X,X)$. \\

Let $A$ be a (nonempty) closed subset of a compact metric space $X$. 
Here we need the following notion:    
$D(A) < \eta$  if $A$ can be decomposed into finitely many mutually disjoint closed sets $A_i$ with diam$(A_i) < \eta$ for each $i$, i.e. $A=\bigcup_i A_i$,  diam$(A_i) < \eta$, and $A_i\cap A_j=\emptyset$ for $i \neq j$. Note that $\dim A=0$ if and only if 
$D(A)< \eta$ for each $\eta >0$.

Modyfying the proof of \cite[Lemma 3.11]{KM20}, we have the following. 

\begin{lem} {\rm (c.f. \cite[Lemma 3.11]{KM20})}  Let $\eta >0$ and $k\in \N$. 
Suppose that $T: X \to X$ is a doubly 0-dimensional map of a
compact metric space $X$
such that $\dim X = d < \infty$ and
$D(\mathrm{cl}[\cup_{p=0}^{4k} T^{-p}(P(T))]) < \eta$. Then  $T$ has $(k,\eta)$-TSP. 
\end{lem}

\begin{proof} 
Since $D(\mathrm{cl}[\cup_{p=0}^{4k} T^{-p}(P(T))]) < \eta$,  there is
an open cover $\mathcal{C} = \{C_i \mid 
1\leq i \leq M \}$ of $X$ such that \\
$(a)$~ $\mathrm{diam} (C_i) < \eta$ for each  $1\leq i \leq M$, and \\
$(b)$~$\mathrm{bd}(C_i)\cap (\mathrm{cl}[\cup_{p=0}^{4k} T^{-p}(P(T)))])=\emptyset$ for each $1\leq i \leq M$. 

Put $K=\bigcup_{i=1}^M\mathrm{bd}(C_i)$. Then by (b) there is an open neighborhood $K'$ of $K$ in $X$ such that for any point $z \in K'$, $T^t(z)\cap T^{t'}(z)=\emptyset$ for $-2k\leq t \leq 2k$.  \\
By modyfying the proof of \cite[Lemma 3.11]{KM20}, we see that 
there is
an open cover  $\mathcal{C}'= \{C'_i \mid 
1\leq i \leq M \}$ of $X$ such that \\ 
$(1)$ $ C'_i \subset C_i$ for each $1\leq i \leq M$, and \\
$(2)$ $\mathrm{ord}\{ f^{-p} (\mathrm{bd}(C'_i)) \mid 1\leq i \leq M, p=0,1,...,k \} \leq d$, and \\
 $(3)$
$\mathrm{bd}(C'_i) \cap (\mathrm{cl}[\cup_{p=0}^{4k} T^{-p}(P(T))]) = \emptyset$ for each $1\leq i \leq M$. \\

 Put $c'_1 = \mathrm{cl}(C'_1),$
$c'_i =
\mathrm{cl}(\mathrm{int} [(C'_i) \setminus (\bigcup_{j <i}C'_{j})])$ 
for $2 \leq i \leq M$. We define 
$$H=\bigcup_{i=1}^M\mathrm{bd}(c'_i)~\mbox{and}~
U_i=\mathrm{int}(c'_i)~(i=1,2,..,M).$$
Then $H$ satisfies the desired conditions of $(k,\eta)$-TSP. 
\end{proof}

\begin{lem}{\rm (A version of Borsuk's homotopy extension theorem, c.f. 
\cite[(8.1)Theorem]{Bor67} and \cite[Theorem 4.1.3]{Mil01}) }
Let $X$ be a compact metric space and $M$ a closed subset of
$X$, and let maps $f', g': M \to \R^k$  satisfy $d(f',g')<\epsilon$. If $g:X\to \R^k$ is an extension of $g'$, then $f'$ has an extension $f:X \to \R^k$ such that $d(f,g)<\epsilon$. 
\end{lem}

Let $X$ be any compact metric space. For each $\alpha >0$ and $S\subset \N$ a set of cardinarity $2d+1$, let   $E(\alpha;S)$ be the subset of
$C(X,X)\times C(X,\R)$ consisting of all pairs $(T,f)$ such that  $I_{T,f}^S: X \to  \R^{S}$ 
is an $\alpha$ {\it trajectory-embedding} (i.e., $I^S_{T,f}(x)\neq I^S_{T,f}(y)$  whenever $x,y \in X$ with $d(T^j(x),T^j(y))\geq \alpha$ for all $j\in S$). 

The main theorem of this paper is the following.

\begin{mainthm} {\rm (Reconstruction theorem of dynamical systems)} Let $X$ be a 
 compact metric space with $\dim X=d$. Suppose that $X$ belongs to the class 0-$\mathcal{DCR}$.Then the following assertions  $(1)-(3)$  hold.\vspace{2mm}\\
$(1)$ {\rm ($\alpha$ trajectory-embedding)}  Let  $\alpha >0$ and $S\subset \N$ a set of cardinarity $2d+1$. Then the set $E(\alpha;S)$ is a dense open set of $C(X,X)\times C(X,\R)$. \vspace{1mm}\\
$(2)$ {\rm (Trajectory-embedding)} There exists a $G_{\delta}$-dense set $E$ of $C(X,X)\times  C(X,\R)$ such that if $(T,f) \in E$, for any $S\subset \N$ of cardinality $2d+1$ 
 $$I_{T,f}^S: X \to  \R^{S}$$ is a trajectory-embedding.  \vspace{1mm} \\
$(3)$  {\rm (Infinite delay observation)} If $E$ is the set as in the above (2), then for any $(T,f) \in E$, 
$$I_{T,f}=T_{T,f}^{\N}:(X,T) \to (\R^{\N},\sigma)$$  satisfies the following conditions: \\
\hspace*{3mm} $(a)$ $I_{T,f}:[O(T)] \cong [O(\sigma_{T,f})]$, \\
\hspace*{3mm} $(b)$ $h(T)=h(\sigma_{T,f})$ and \\
\hspace*{3mm} $(c)$ if $x, y\in X$ are trajectory-separated for $T$, then 
$$|\{i\in \N|~I_{T,f}(x)_i=I_{T,f}(y)_i\}|\leq 2d.$$

In particular, $X$ is a reconstruction space of dynamical systems. 
\end{mainthm}

\begin{proof}
 We prove (1). Let $\alpha >0$ and $S \subset \N$ of cardinality $2d+1$.  
 For each $T\in C(X,X)$,  we put
$$L(T:\alpha,S) =$$ $$\{(x,y)\in X\times X |~d(T^{j}(x),T^{j}(y))\geq \alpha~\mbox{for}~j \in S\} \subset X\times X.$$
Recall the set  
$$E(\alpha;S)=$$ $$\{(T,f)\in C(X,X)\times C(X,\R)|I_{T,f}^S(x)\neq I_{T,f}^S(y)~\mbox{for}~(x,y)\in L(T:\alpha,S) \}. $$  

We will show that $E(\alpha,S)$ is an open subset of $C(X,X)\times C(X,\R)$. Let 
$(T,f)\in E(\alpha,S)$. Since $L(T:\alpha,S)$ is compact, we can choose a neighborhood $K$ of $L(T:\alpha,S)$ and  $\epsilon >0$ such that for any $(x,y)\in K$,
$$d(I_{T,f}^S(x),I_{T,f}^S(y))\geq 2\epsilon.$$ 
Note that the $$L(\alpha,S): C(X,X)\to 2^{X\times X}\cup \{\emptyset\}, T \mapsto 
L(T:\alpha,S)$$  is an upper semi-continuous set-valued function. We can choose a neighborhood $U(T)$ of $T$ in $C(X,X)$ and a neighborhood
$V(f)$ of $f\in C(X,\R)$ such that if $(T',f') \in U(T) \times V(f)$, then 
$L(T':\alpha,S) \subset K$ and for $(x,y) \in K$, 
$$d(I_{T',f'}^S(x), I_{T',f'}^S(y))\geq \epsilon.$$
Since $L(T':\alpha,S) \subset K$,  we see that
$I_{T',f'}^S(x)\neq I_{T',f'}^S(y)$  
for $(x,y) \in L(T':\alpha,S)$. Then $(T',f') \in E(\alpha,S)$ and so 
$U(T)\times V(f) \subset E(\alpha,S)$. Hence $E(\alpha,S)$ is an open set
of $C(X,X)\times C(X,\R)$. 

Next, we will show that $E(\alpha,S)$ is dense in $C(X,X)\times C(X,\R)$. 
 Let $(T,f) \in C(X,X)\times C(X,\R)$ and $\epsilon >0$. 
 Since $f:X\to \R$ is uniformly continuous, there is a sufficiently small positive number 
 $\eta >0$ such that $\eta < \alpha$ and if $x,y\in X$ with $d(x,y)<\eta$, 
  then $d(f(x),f(y))<\epsilon$. Let $k=\max S$.  
 By Remark 2, we can choose $T_1 \in C(X,X)$ such that  such that $T_1$ is  a 
 0-dimensional map, $d(T,T_1)<\epsilon/2$ and $\dim CR(T_1)=0$.  Since  $\dim (\cup_{p=0}^{4k} T_1^{-p}(CR(T_1)))=0$, we choose a closed neighborhood $W$ of 
  $\cup_{p=0}^{4k} T_1^{-p}(CR(T_1))$ in $X$ with $D(W)<\eta$. Since the set function $CR: C(X,X) \to 2^X$ is upper semi-continuous and $X$ satisfies the condition 0-$\mathcal{D}$ of Definition 5.1, 
  we can choose a dubly 0-dimensional map $T_2\in C(X,X)$ such that $d(T_1,T_2)< \epsilon/2$ and $\cup_{p=0}^{4k} T_2^{-p}(CR(T_2))\subset W$. Then   $D(\cup_{p=0}^{4k} T_2^{-p}(CR(T_2))) < \eta$ and so $D(\mathrm{cl}[\cup_{p=0}^{4k} T_2^{-p}(P(T_2))]) < \eta$.  By Lemma 5.2, we see that $T_2$ has $(k,\eta)$-TSP.   Hence 
there is a closed set $H$ of $X$
such that \\
$(1)$ $X\setminus  H$ is a union of finitely many disjoint open sets
 of  diameter at most $\eta$,  and \\
$(2)$ $\mathrm{ord}\{T_2^{-p}(H)\}_{p=0}^k \leq d$.   

We choose a small open neighborhood $G$ of $H$ in $M$  such that \\
 $ (2')$ $\mathrm{ord}\{T_2^{-j}(G)\}_{j=0}^k \leq d.$\\
Then we may assume that $X\setminus \mathrm{cl}(G)$ is a union of  disjoint open sets
$V_i (i = 1,2,...,m)$ such that $\mathrm{cl}(V_i) \subset U_i$.  
 Note that
$ \mathrm{cl}(V_i) \cap \mathrm{cl}(V_j)=\emptyset (i \neq j)$. For each $i$, take a point $t_i$ which belongs to a sufficiently
small neighborhood of $f(\mathrm{cl}(V_i))$ in $\R$ such that $t_i\neq t_j$ if $i\neq j$. We define a map
$$g': \bigcup_{i=1}^m \mathrm{cl}(V_i) \to \R$$
by $g'(\mathrm{cl}(V_i))=t_i$. Then by Lemma 5.3, we have an extension
$g: X \to \R$ of $g'$ with $d(g,f)<\epsilon$.
We will prove $(T_2,g) \in E(\alpha,S)$. Let $(x,y) \in L(T_2:\alpha,S)$. By (d'),
$$|\{j \in S|~T_2^{j}(x) \in G\}|\leq d$$ and 
$$|\{j \in S|~T_2^{j}(y) \in G\}|\leq d.$$ 
Since $|S|=2d+1$, we can find some $j \in S$  such that $T_2^{j}(x), T_2^{j}(y) \in X\setminus G$. 
Since $d(T_2^{j}(x),T_2^{j}(y)) \geq \alpha$ and $\mathrm{diam}(\mathrm{cl}(V_i)) <\eta <\alpha$ for each $i=1,2,..,m$, 
there are $n, n'$ such that $n\neq n'$ and $T_2^{j}(x) \in \mathrm{cl}(V_n)$ and 
$T_2^{j}(y) \in \mathrm{cl}(V_{n'})$.
Then $gT_2^{j}(x)=t_n \neq t_{n'}=gT_2^{j}(y)$.  This implies $I_{T_2,g}^S(x)\neq I_{T_2,g}^S(y) $ 
 and hence $$(T_2,g)\in E(\alpha,S).$$
Note that $d(T,T_2) <\epsilon$ and $d(f,g) < \epsilon$. So we see that $E(\alpha,S)$ is a dense open set of
$C(X,X)\times C(X,\R)$.

We will prove $(2)$. Let $J$ be the set of all set $S \subset \N$ of cardinality $(2d+1)$. 
 Note that $J$ is a countable set. We define 
$$E=\bigcap\{E(1/n,S)|~S\in J ~\mbox{and} ~n\in \N\setminus \{0\} \}.$$
Then we see that $E$ is a desired dense $G_{\delta}$-set in $C(X,X)\times C(X,\R)$.

Finally we will prove (3). Let $(T,f)\in E$. Note that if $k=2d$, then $(T,f)$ satisfies the conditions of Proposition 4.1.  Hence $$I_{T,f}:[O(T)] \cong [O(\sigma_{T,f})] ~\mbox{and}~h(T)=h(\sigma_{T,f}).$$ Let  $x, y\in X$ be trajectory-separated points for $T$. Suppose, on the contrary, that $$|\{i\in \N|~I_{T,f}(x)_i=I_{T,f}(y)_i\}|> 2d.$$ Then we can choose a set $S' \subset \{i\in \N|~I_{T,f}(x)_i=I_{T,f}(y)_i\}$ with $|S'|=2d+1$. This is a contradiction to the fact that  $I_{T,f}^{S'}$ is  a trajectory-embedding.

This completes the proof.

\end{proof}

\section{The class 0-$\mathcal{DCR}$} 
In this section, we consider  the following general problem. 

\begin{prob}
What kinds of  compact metric spaces belong to the class 0-$\mathcal{DCR}$ ?
\end{prob}

We will show that PL-manifolds, some branched manifolds and some fractal sets, e.g. Menger manifolds, Sierpi\'{n}ski carpet, Sierpi\'nski  gasket and dendrites, belong to the class 0-$\mathcal{DCR}$.  

In \cite{KOU16} Krupski, Omiljanowski and
Ungeheuer defined the class 0-$\mathcal{CR}$ which is the family of all compact metric spaces $X$ such that the set $CR(T)$ is 0-dimensional for a generic map $T\in C(X,X)$. 
They proved the following result. 
 
\begin{thm}{\rm (\cite[Theorem 5.1]{KOU16})} If $X$ is a (compact) polyhedron, then  $X\in 0$-$\mathcal{CR}$. Moreover,  if $X$ is a compact metric space that admits an $\epsilon$-retraction $r_{\epsilon}:X\to P$ onto a polyhedron $P\subset X$ for each $\epsilon >0$ (i.e., $d(r_{\epsilon}, id_X)<\epsilon$ and $r_{\epsilon}|P=id_P$), 
then $X\in 0$-$\mathcal{CR}$. 
\end{thm}

Now, we will consider  the family  0-$\mathcal{D}$ of all compact metric spaces $X$ such that all doubly 0-dimensional maps on $X$ is dense in $C(X,X)$. 
A map $T:X \to X$ is said to be  a {\it piecewise embedding} if there is a countable family $\{F_i\}_{i\in \N}$ of closed subsets of $X$ such that $X=\bigcup_{i\in \N} F_i$ and $T|F_i:F_i\to X$ is injective for each $i\in \N$. Note that if  a map $T:X \to X$ is a piecewise embedding, then $T$ is doubly 0-dimensional because that  $\dim T^{-1}(x)$ is a countable set for each $x \in X$ and  $$\dim T(A)=\max\{\dim T(A\cap F_i)~|~i\in \N\}\leq 0$$ for any 0-dimensional closed set $A$ of $X$ (see the countable sum theorem for dimension  \cite[Theorem 3.1.8]{Eng95}).

A (compact) $d$-dimensional  polyhedron $P~(d\geq 1)$ is  called a {\it manifold with branch structures} if $P=\bigcup_{j\in J}M_j\cup M$, where \\ 
$(1)$ $\{M_j\}_{j\in J}$ is a finite family of  mutually disjoint closed sets  of $P$ such that   for each $j\in J$, 
$$M_j=N_j \cup_{\varphi_{\alpha}}\bigcup \{N_{j,\alpha}|\alpha \in J_j\},$$ where  
 $J_j$ is a finite set, $N_j, N_{j,\alpha}~(\alpha \in J_j)$ are $d$-dimensional manifolds with boundaries, and $M_j$ is obtained from $N_j$ by attaching $N_{j,\alpha}~(\alpha \in J_j)$ via locally embedding maps $\varphi_{\alpha}: N'_{j,\alpha}  \to \partial N_j$ from a $(d-1)$-dimensional (compact) submanifold  $N'_{j,\alpha}$ of  $\partial N_{j,\alpha}$ into $\partial N_j$, i.e., $M_j$ is the quotient space of the topological sum $N_j \amalg_{\alpha\in J_j} N_{j,\alpha}$ under the identifications $ x \sim \varphi_{\alpha}(x)$ for $x\in N'_{j,\alpha}\subset \partial N_{j,\alpha}$ and  the quotient map is denoted by $q_j: N_j \amalg_{\alpha\in J_j} N_{j,\alpha} \to M_j~(=N_j \cup \bigcup\{q_j(N_{j,\alpha})|~\alpha \in J_j\})$, \\
$(2)$ $M$ is a $d$-dimensional compact manifold in $P$  with $$M\cap \bigcup\{\varphi_{\alpha}(N'_{j,\alpha})~|~ j\in J, \alpha \in J_j\}=\emptyset$$ and \\
$(3)$  $P\setminus  \bigcup\{\varphi_{\alpha}(N'_{j,\alpha})~|~ j\in J, \alpha \in J_j\}$ is a $d$-dimensional (non-compact) manifold.\\

Remark. All PL-manifolds and some branched manifolds are  manifolds with branch structures. The associated template of the well-know Lorenz attractor is a manifold with branch structures \cite{GL02}. \\

Let $K$ be a simplicial complex and let $K^{(m)}$ be the $m$-skeleton of $K$, i.e., the set of all simplexes of $K$ whose dimension are $\leq m$. For a vertex $v$ of $K^{(0)}$, let $\mathrm{St}(v,K)$ be the closed star of $v$, i.e., $\mathrm{St}(v,K)=\bigcup\{\sigma \in K|~v \in \sigma \}$. Also 
let $\beta K$ denote the barycentric subdivision of $K$. Let $\Delta=<p_0,p_1,\cdots,p_n>$ and $\sigma=<v_0,v_1,...,v_n>$ be $n$-simplexes and let $J$ be the set of all sequence $\{\star\}=s_0, s_1, \cdots, s_n=\sigma$ of faces of $\sigma$ such that $s_{i-1}$ is a face of $s_i$ and $\dim s_{i-1}+1=\dim s_i$  for $i=1,2,...,n$. Then $|J|=(n+1) !$ and 
$$\sigma=\bigcup\{<b(s_0),b(s_1),\cdots,b(s_n)>|~(s_0,s_1,\cdots ,s_n) \in J \},$$ where $b(s_i)$ is  the barycenter of $s_i$. Consider the {\it folding map (at barycenters)} 
 $f_{\sigma}:|\beta\sigma| \to \Delta$ which is the simplicial map defined by $f_{\sigma}(b(s_i))=p_i$ for each $i=0,1,2,...,n$. Note that  $f_{\sigma}$ is a piecewise embedding. 

\begin{prop}
Let $P$ be a manifold with branch structures. Then the set of all piecewise embedding maps $T: P\to P$ is dense in  $C(P,P)$. In particular, $P$ belongs to 
 0-$\mathcal{DCR}$. Hence $P$ is a reconstruction space of dynamical systems. 
\end{prop} 

\begin{proof}  Let $\dim P=d\geq 1$. Since $P$ is a polyhedron, by Theorem 6.2, $P$ belongs to  the class 
0-$\mathcal{CR}$. We will show that $P$ belongs to the class 0-$\mathcal{D}$. 
Let $T\in C(P,P)$ and $\epsilon>0$. We choose a simplicial complex $K$ of $P$ such that   $\mathrm{mesh}(K)$  is sufficiently small, i.e., $\mathrm{mesh}(K)<\epsilon/2$. Take a simplicial approximation $T_1: P=|L|\to |K|$ of $T$ such that 
$d(T,T_1)<\epsilon/2$, where $L$ is a subdivision of $K$. 

By modifying $T_1$, we will construct a map $T'_1:|\beta L|\to P$ such that 
for each $d$-simplex $s$ of $\beta L$, $T'_1|s:s\to P$ is an embedding and $d(T_1,T_1')<\epsilon$. We consider the following abstract simplicial complex $\tilde{K}$ which contains the  simplicial complex $K$ as follows: For each $0\leq k\leq d-1$, let $$J_k=\{(a_0,a_1,..,a_k)\in \N^{k+1}~|~d=k+\Sigma_{i=0}^k a_i\}.$$ For each $k$-simplex $\sigma=<v_0,v_1,...,v_k> ~(k\leq d-1)$ of $K$ and  each $(a_0,a_1,..,a_k)\in J_k$, we consider the abstract $d$-simplex 
$$<v_0,v_1,...,v_k;(a_0,a_1,..,a_k)>$$  $$=<p_{(v_0,0)},p_{(v_0,1)}
,..,p_{(v_0,a_0)},p_{(v_1,0)},p_{(v_1,1)},..,p_{(v_1,a_1)},\cdots,p_{(v_k,0)},p_{(v_k,1)}..,p_{(v_k,a_k)}>$$ 
where we assume $v_{i}=p_{(v_i,0)}\in K^{(0)}~(i=0,1,...,k)$. In particular, 
 $$<v;d>=<p_{(v,0)},p_{(v,1)},...,p_{(v,d)}>$$
 for each vertex $v \in K^{(0)}$, where $v=p_{(v,0)}$.  We define the abstract simplicial complex $\tilde{K}$ as follows: 
$$\tilde{K}=K\bigcup \{s~| s ~\mbox{is a face of }<v_0,v_1,...,v_k;(a_0,a_1,..,a_k)>, 0\leq k\leq d-1, $$ $$ (a_0,a_1,..,a_k)\in J_k, 
\mbox{and} <v_0,v_1,...,v_k> \in K^{(k)}\setminus K^{(k-1)}\}.$$ 
For each $0\leq k \leq d-1$, we put  
$$A_k=\bigcup\{<v_0,v_1,...,v_k;(a_0,a_1,..,a_k)>| <v_0,v_1,...,v_k>\in K^{(k)}\setminus K^{(k-1)}, $$ $$ (a_0,a_1,..,a_k)\in J_k\}.$$

We will construct a retraction $r:|\tilde{K}|\to |K|$ such that $$r|<v_0,v_1,...,v_k;(a_0,a_1,..,a_k)>$$ is injective. Recall that $P=|K|$ is a manifold with branch structures.
So we assume that $$P= \bigcup_{j\in J}M_j \cup M, M_j=N_j \cup_{\varphi_{\alpha}}\bigcup \{N_{j,\alpha}|\alpha \in J_j\},$$
$N'_{j,\alpha}, \varphi_{\alpha}$ and $q_j$ are defined as above.

By induction on $k~(0\leq k\leq d-1)$,  we will construct $h_k:\cup_{i=0}^kA_i \to |K|$.
First we will construct a map $h_0:A_0\to |K|$ as follows.    
Let $v\in K^{(0)}$. 

 If $v \notin  \bigcup\{\varphi_{\alpha}(N'_{j,\alpha})~|~ j\in J, \alpha \in J_j\},$   
we choose an embedding $$h_0:<v;d>(=<p_{(v,0)},p_{(v,1)},...,p_{(v,d)}>)\to P\setminus  \bigcup\{\varphi_{\alpha}(N'_{j,\alpha})~|~ j\in J, \alpha \in J_j\}$$ 
 with $h_0(v)=v$, because that $P\setminus  \bigcup\{\varphi_{\alpha}(N'_{j,\alpha})~|~ j\in J, \alpha \in J_j\}$ is a $d$-dimensional (non-compact) manifold. 
 
If $v\in \varphi_{\alpha}(N'_{j,\alpha})$ for some $j\in J, \alpha \in J_j$, we choose an embedding $$h_0:<v;d> \to N_j$$  with $h_0(v)=v$, because that $N_j$ is a $d$-dimensional manifold.
So we have a map $h_0:A_0\to |K|$. 

Now we assume that 
$h_{k-1}:\cup_{i=0}^{k-1}A_{i}\to |K|$ have been constructed. 
Let $<v_0,v_1,...,v_k>$ be a $k$-simplex of $K$. 

If  $<v_0,v_1,...,v_k>$ is contained in $P\setminus  \bigcup\{\varphi_{\alpha}(N'_{j,\alpha})~|~ j\in J, \alpha \in J_j\}$, then we can choose an embedding 
  $$h_k:<v_0,v_1,...,v_k;(a_0,a_1,..,a_k)> \to P\setminus  \bigcup\{\varphi_{\alpha}(N'_{j,\alpha})~|~ j\in J, \alpha \in J_j\}$$  satisfying the following conditions $(a)$ and $(b)$:\\
 $(a)$ $h_k|<v_0,v_1,...,v_k>=id$ and \\
 $(b)$  
$$h_k|A_{k-1}\cap <v_0,v_1,...,v_k;(a_0,a_1,..,a_k)>= $$ $$h_{k-1}|A_{k-1}\cap <v_0,v_1,...,v_k;(a_0,a_1,..,a_k)>.$$ 

If $<v_0,v_1,...,v_k>$ is contained in $N_j$ for some $j\in J$ and 
$$<v_0,v_1,...,v_k>\cap \bigcup\{\varphi_{\alpha}(N'_{j,\alpha})~|~ j\in J, \alpha \in J_j\}\neq \emptyset,$$
 then we choose  an embedding  $h_k:<v_0,v_1,...,v_k;(a_0,a_1,..,a_k)> \to N_j$  satisfying  $(a)$ and $(b)$ as above. 

 If  $<v_0,v_1,...,v_k>$ is contained in $q_j(N_{j,\alpha})$ and  $<v_0,v_1,...,v_k>$  intersects with  $\partial N_j$,  then we choose an embedding $$h_k:<v_0,v_1,...,v_k;(a_0,a_1,..,a_k)> \to  N_j \cup q_j(N_{j,\alpha})$$ satisfying  $(a)$ and $(b)$ as above,  
 because that as the assumption of the case $k-1$, we can assume that $$h_{k-1}( A_{k-1}\cap \partial<v_0,v_1,...,v_k;(a_0,a_1,..,a_k)>)$$ is contained in a $d$-dimensional manifold in $N_j \cup q_j(N_{j,\alpha})$.

By induction on $k$, we obtain $h_{d-1}$. And by use of $h_{d-1}$ we have a retraction $r:|\tilde{K}|\to |K|$ such that $r|<v_0,v_1,...,v_k;(a_0,a_1,..,a_k)>$ is injective. 

Next, we will define a PL-map $\varphi: |L| \to |\tilde{K}|$ which is a piecewise embedding.
For  each $d$-simplex $\sigma$ of $L$, we consider the simplex 
$$T_1(\sigma)=<v_0,v_1,...,v_k> \in K~(k\leq d).$$ 
For each vertex $v_i$ of $T_1(\sigma)$, we consider the face 
$$T^{-1}(v_i)\cap \sigma=<w_{(i,0)},w_{(i,1)},..,w_{(i,a_i)}>=\sigma_{v_i}$$ of $\sigma$. Note that 
$d=k+\Sigma_{i=0}^k a_i$ and 
$$\sigma=<w_{(v_0,0)},..,w_{(v_0,a_0)},w_{(v_1,0)},..,w_{(v_1,a_1)},\cdots,w_{(v_k,0)},..,w_{(v_k,a_k)}>$$
$$\equiv \sigma_{v_0}*\sigma_{v_1}*\cdots *\sigma_{v_k}.$$
We put $$ \beta\sigma_{v_0}*\beta\sigma_{v_1}*\cdots *\beta\sigma_{v_k}=\{\tau_0*\tau_1\cdots *\tau_k|~\tau_i\in \beta\sigma_i, \dim \tau_i=a_i\}.$$ 
Then $ \beta\sigma_{v_0}*\beta\sigma_{v_1}*\cdots *\beta\sigma_{v_k}$ gives a subdivision of $\sigma$. 
 Consider the (abstract) $d$-simplex $$\Delta_{\sigma}=<v_0,v_1,...,v_k;(a_0,a_1,..,a_k)>$$ $$=<p_{(v_0,0)},..,p_{(v_0,a_0)},p_{(v_1,0)},..,p_{(v_1,a_1)},\cdots,p_{(v_k,0)},..,p_{(v_k,a_k)}>$$ 
of $\tilde{K}$  
 and consider the folding map 
$$f_{\sigma_{v_i}}: |\beta\sigma_{v_i}| \to \Delta_{v_i}=<p_{(v_i,0)},..,p_{(v_i,a_i)}>~(\in \tilde{K})$$ defined as above.

For each  $d$-simplex $\sigma$ of $L$, we have a map  
$$\varphi_{\sigma}=f_{\sigma_{v_0}}* f_{\sigma_{v_1}} \cdots * f_{\sigma_{v_k}}: $$ $$\sigma= |\beta\sigma_{v_0}*\beta\sigma_{v_1}*\cdots *\beta\sigma_{v_k}| \to \Delta_{v_0}*\Delta_{v_1}*\cdots *\Delta_{v_k}=\Delta_{\sigma} \in \tilde{K}.$$ 
Note that if $\dim T_1(\sigma)=d$, $\varphi_{\sigma}=T_1|\sigma$.

By use of 
$\varphi_{\sigma}$, we have a desired PL map $\varphi: |L| \to |\tilde{K}|$ which is a piecewise embedding. 
Finally, we put $T_1'=r\varphi: P\to P$. Then $T_1'$ is a piecewise embedding. Also by the constraction of $r$, we may assume that $d(T_1,T_1')<\epsilon/2$. This means that $P$ satisfies the condition (0-$\mathcal{D})$. 
 This completes the proof. 
\end{proof}

Many  dynamical properties of Cantor sets have been studied by many authors. Now we consider dynamical properties of higher dimensional fractal sets. 

 For $0\leq k < n$, we will construct  a space $L^n_k$  in the $n$-simplex $M_0=<v_0,v_1,...,v_n>$ by Lefshetz's method (see \cite[p.129]{Chi96} and \cite{Lef31}). 
We define a sequence $\{(M_i,L_i)\}_{i\in \N}$ of compact $n$-dimensional polyhedra $M_i$ with triangulations $L_i$ inductively as follows. 
Let $M_0$ be the $n$-simplex $<v_0,v_1,...,v_n>$ with the standard simplicial complex structure $L_0$. 
Suppose $(M_i,L_i)$ has been defined. Let 
$$M_{i+1}=\bigcup\{{\mathrm St}(v,\beta^2(L_i))~|~v \mbox{ is a vertex of }~ \beta(L_i^{(k)}) \}$$ and $$L_{i+1}=\beta^2L_{i}|M_{i+1}.$$  Note that $M_{i+1}$ may be regarded as a regular neighborhood of the $k$-skeleton of $L_i$. 
Then $\{M_i\}_{i\in \N}$ is a decreasing sequence and we obtained a compact metric space $$L^n_k=\bigcap_{i\in \N}M_i.$$
 Note that $L^1_0$ is a Cantor set and $L^{2d+1}_d~(=\mu^d)$ is called  the $d$-dimensional {\it Menger compactum}. Also $L^2_1$ is called the {\it Sierpi\'{n}ski carpet}. 
A space $X$ is a $d$-dimensional {\it Menger manifold} if $X$ is compact and each point $x$ of $X$ has a neighborhood $W$ of $x$ in $X$ such that $W$ is homeomorphic to the $d$-dimensional Menger  compactum $\mu^d$ (for many geometric properties of $\mu^d$, see \cite{Bes88}).

Also the  Sierpi\'{n}ski gasket can be constructed from an equilateral triangle by repeated removal of (open) triangular subsets: Start with an equilateral triangle. Subdivide it into four smaller congruent equilateral triangles and remove the central (open) triangle.
Repeat this step with each of the remaining smaller triangles infinitely. So we have a sequence $\{X_i\}_{i\in \N}$ of continua in the plane and the intersection $X=\bigcap_{i\in \N}X_i$ is called the {\it Sierpi\'{n}ski gasket}. 

A compact connected metric space (=continuum) $X$ is said to be a {\it dendrite} if $X$ is a 1-dimensional locally connected continuum which contains no simple closed curve.

\begin{prop} Let $M$ be a $d$-dimensional Menger manifold. Then  $M$ belongs to  0-$\mathcal{DCR}$ and hence $M$ is a reconstruction space. More precisely, there exists a $G_{\delta}$-dense set $E'$ of $C(M,M)\times  C(M,\R)$ such that if $(T,f) \in E'$, then for any $S\subset \N$ of cardinality $2d+1$, $I_{T,f}^S: M \to  \R^{S}$ is an embedding and so  
 $$I_{T,f}=T_{T,f}^{\N}:(M,T) \to (\R^{\N},\sigma)$$  is an embedding. 
\end{prop} 

\begin{proof} By  \cite[Definition 1.2.1 and Corollary 5.2.2]{Bes88},  for each $\epsilon >0$, $M$ admits an $\epsilon$-retraction $r_{\epsilon}:M\to P$ onto a $d$-dimensional polyhedron $P\subset M$. Hence by Theorem 6.2, $M$ belongs to 0-$\mathcal{CR}$. Also it is well-known that the set $e(M,M)$ of all embeddings $T: M\to M$ is  a $G_{\delta}$ dense set of $C(M,M)$ (see  \cite[Theorem 2.3.8]{Bes88}). Hence 
$M$ belongs to  0-$\mathcal{DCR}$. Recall the proof of Theorem 5.4. We can complete the latter part of the proof by replacing $C(M,M)$  with $e(M,M)$. 
\end{proof}

We will show that the  Sierpi\'{n}ski carpet belongs to  0-$\mathcal{DCR}$. In \cite[p.323]{Why58}, Whyburn proved that the  Sierpi\'{n}ski carpet is homeomorphic to any $S$-curve $X$ (=plane locally connected 1-dimensional continuum whose complement  in the plane consists of countably many components with frontiers being mutually disjoint  simple closed curves $\{S_i\}_{i\in \N}$, and moreover, if $K_1, K_2$ are $S$-curves and $C_1, C_2$ are frontiers of  components of complements of $K_1, K_2$ in the plane $\R^2$, respectively, then each homeomorphism of $C_1$ onto $C_2$ can be extended to a homeomorphism of $K_1$ onto $K_2$. 
Such simple closed curves $\{S_i\}_{i\in \N}$ are called the {\it  rational circles} of  the $S$-curve $X$. The union of all these circles $\{S_i\}_{i\in \N}$ is called the {\it rational part} of  $X$, and the remainder $X\setminus (\bigcup_{i\geq 0} S_i)$ is called the  {\it irrational part} of $X$. We need the following lemma. 

\begin{lem}
Let $X$ be an $S$-curve in the plane $\R^2$ and let  $\{S_i\}_{i\in \N}$ be  rational circles of  $X$, and $S_0$ the frontier of the unbounded component of $\R^2 \setminus X$. Let $B_k~(k\geq 1)$ be the disk in $\R^2$ with $\partial B_k=S_k$. If $p:\R^2 \to H$ is 
 the decomposition map of $\R^2$ obtained by identifying the sets $B_1, B_2,...$ to single points respectively, then the decomposition space $H$ is homeomorphic to $\R^2$, $p(X)=D$ is a disk in the plane $H$ with $\partial D=p(S_0)$, and 
 the set $\{p(S_i)~|i=1,2,..\}$ is a countable set in  $D\setminus \partial D$. Moreover, for a point $x$ of $X\setminus S_0$, $x$ is in the  irrational part of $X$ if and only if $p^{-1}(p(x))=\{x\}$. 
\end{lem}

\begin{proof} By the Moore's theorem \cite[p.380]{Kur61}, we see that 
$H$ is homeomorphic to $\R^2$ and $p(X)=D$ is a disk. Note that the set $\{p(S_j)| j\geq 1\}$ is a countable set in the disk $D$.
\end{proof}

\begin{prop} Let $X=L^2_1\subset \R^2$ be the  Sierpi\'{n}ski carpet. Then  $X$ belongs to  0-$\mathcal{DCR}$. 
\end{prop} 

\begin{proof} Let $\epsilon >0$.  Recall the Lefshetz's construction of 
$L^2_1$ as above.  We see that $M_{i+1}$ is regarded as a regular neighborhood of the $1$-skeleton of $L_i$. So we can easily see that  $X$ admits an $\epsilon$-retraction $r_{\epsilon}:X\to |L_i^{(1)}|$ for a sufficiently large $i\in \N$.  Hence $X$ belongs to 0-$\mathcal{CR}$. 

We will show that $X$ belongs to the class 0-$\mathcal{D}$. Let $T\in C(X,X)$ and  $\epsilon >0$.   Let $M_0=\Delta_2$ be a 2-simplex with the standard simplicial complex structure $L_0$.  
We have the sequence $(M_i,L_i)$ defined as above, i.e.,     
$$M_{i+1}=\bigcup\{{\mathrm St}(v,\beta^2(L_i))~
|~v \mbox{ is a vertex of }~ \beta(L_i^{(1)}) \}$$ and $$L_{i+1}=\beta^2L_{i}|M_{i+1}.$$ 
Then 
$L^2_1=\bigcap_{i\in \N}M_i.$
Note that 
${\mathrm St}(v,\beta^2(L_i))$ is a disk in $\R^2$ and  
 ${\mathrm St}(v,\beta^2(L_i))\cap X$ is an $S$-curve. Choose  a sufficiently large natural number $i_0$ so that 
 $\mathrm{diam ~St}(v,\beta^2(L_{i_0}))<\epsilon$ for each vertex $v$ of 
 $\beta(L_{i_{0}}^{(1)})$. Put $D_v=\mathrm{St}(v,\beta^2(L_{i_0}))$. 

Let $\{S_k\}_{k\in \N}$ be the family of rational circles of the $S$-curve $X$ such that  
$S_k \subset \mathrm{int}_{\R^2}M_{i_0+1}$ and  $B_k$ is the disk with $\partial B_k=S_k$  for each $k\in \N$.   Let $p:\R^2 \to H$ be  
 the decomposition map of $\R^2$ obtained by identifying the sets $B_1, B_2,...$ to single points respectively. Then  $p(D_v)$ is a disk in the plane $H$. So   we have a family $$\{p(D_v)~|~v \mbox{ is a vertex of }~ \beta(L_{i_0}^{(1)}) \}$$ of disks in $H$  such that  $$p(M_{i_0+1})=\bigcup \{p(D_v)~|~v \mbox{ is a vertex of }~ \beta(L_{i_0}^{(1)}) \}$$  and 
$\mathrm{ord}(\{p(D_v)~|~v \mbox{ is a vertex of }~ \beta(L_{i_0}^{(1)}) \})\leq 2$. Since the set $Z=\{p(S_k)~|k\in \N\}$ is a countable set in $p(M_{i_0+1})$,  we have a disk $E_v$ in $ p(M_{i_0+1})$ such that  \\
$(1)$ $\partial E_v \cap Z=\emptyset$, \\
$(2)$ $p(D_v)\subset \mathrm{int}_{p(M_{i_0+1})}E_v$, \\
$(3)$ $\mathrm{diam}~p^{-1}(E_v)<\epsilon$ and \\
$(4)$ $\mathrm{ord}\{E_v|~v \mbox{ is a vertex of }~ \beta(L_{i_0}^{(1)}) \} \leq 2$. 

If necessary, by use of homeomorphism of $\R^2$ we may assume that $H=\R^2$ and each $E_v$ is a convex set in $H$.  Put $D'_v=p^{-1}(E_v)\cap X$. 
Choose a large natural number  $j_0\geq i_0$ such that for each 2-simplex $\sigma$ of $L_{j_0}$, there is a vertex $v$ of  $\beta(L_{i_0}^{(1)})$ such that  $T(\sigma \cap X) \subset  \mathrm{int}_{X}D'_v$.  For  each $w \in  L_{j_0}^{(0)}$, 
we put  $$V(w)=\{v |~v \mbox{ is a vertex of }~ \beta(L_{i_0}^{(1)}) 
~\mbox{and}~ T(w)\in D'_v\}.$$  
Note that $1\leq |V(w)|\leq 2$. Since $Z$ is a countable set in $H=\R^2$ and by use of usual general position arguments in the plane, we see that for any  $w \in  L_j^{(0)}$, take a point ${\tilde w}$ of the irrational part of $\bigcap\{D'_v|~v \in V(w)\}$ such that  
\\
$(5)$ if $w, w' \in  L_j^{(0)}$ and $w\neq w'$, then ${\tilde w}\neq {\tilde w'}$, \\
$(6)$ the set $\{p({\tilde w}) ~|w\in  L_j^{(0)}\}$ is in general position of the plane $H=\R^2$
and the segment $[p({\tilde w}), p({\tilde w'})]$ in $\R^2$ contains no point of $Z$.\\

Let $\sigma$ be any 2-simplex in $L_{j_0}$  with $\sigma=<w_0,w_1,w_2>$. Consider the 2-simplex $\tilde{\sigma}$ in $H=\R^2$ such that $\tilde{\sigma}^{(0)}=\{p(\tilde{w}_0),p(\tilde{w}_1),p(\tilde{w}_2)\}$. Consider a natural homeomorphism $h_{\sigma}: \partial  \sigma \to p^{-1}(\partial \tilde{\sigma})$ with $h_{\sigma}(w_i)=\tilde{w}_i$. Since $\sigma\cap X$ and $p^{-1}(\tilde{\sigma})\cap X$ are $S$-curves, by Whyburn theorem as above there is a homeomorphism  
$\varphi_{\sigma}: \sigma \cap X \to p^{-1}(\tilde{\sigma})\cap X$ which is an extension of $h_{\sigma}$. By use of $\varphi_{\sigma}$,  we have a desired piecewise embedding $T':X\to X$ with $d(T,T')<\epsilon$. 
Hence 
$X$ belongs to  0-$\mathcal{D}$. 
\end{proof}

\begin{prop} Let $X$ be the Sierpi\'{n}ski gasket. Then  $X$ belongs to  0-$\mathcal{DCR}$. 
\end{prop} 
\begin{proof} Let $\epsilon >0$. 
We see that $X$ admits an $\epsilon$-retraction $r_{\epsilon}:X\to P$ onto a subgraph  $P$ of  $X$ and so $X$ belongs to 0-$\mathcal{CR}$. We will show that $X$ belongs to the class 0-$\mathcal{D}$. 
Let $T\in C(X,X)$ and $\epsilon >0$. Since $T$ is uniform continuous, we choose a sufficiently small positive number  $0<\delta<\epsilon$ such that $d(T,Tr_{\delta}) <\epsilon$, where $r_{\delta}$ is a $\delta$-retraction.  Note that $X$ is a countable union of segments $J_n~(n\in \N)$ in $\R^2$ and also we can choose such a retraction $r_{\delta}$ such that  $r_{\delta}|J_n$ is injective, and hence 
 it is a doubly 0-dimensional map. Consider the map 
$r_{\delta}T|P:P\to P$. Since $P$ is a graph and hence it is a 1-dimensional manifold with branch structures, 
 we have a piecewise embedding map $g:P\to P$ such that $d(g,r_{\delta}T|P)<\epsilon.$ Then 
$$d(T,gr_{\delta})\leq d(T,Tr_{\delta})+d(Tr_{\delta},r_{\delta}Tr_{\delta})+d(r_{\delta}Tr_{\delta},gr_{\delta})<3\epsilon$$ and $gr_{\delta}$ is a doubly 0-dimensional map.  
Hence 
$X$ belongs to  the class 0-$\mathcal{DCR}$. 
\end{proof}

\begin{prop} Let $X$ be any dendrite. Then  $X$ belongs to  0-$\mathcal{DCR}$. 
\end{prop} 

\begin{proof} Since $X$ is a dendrite, we see that  for each $\epsilon >0$, $X$ admits an $\epsilon$-retraction $r_{\epsilon}:X\to P$ onto a subtree $P$ of  $X$. Hence $X$ belongs to 0-$\mathcal{CR}$ (see also \cite{KOU16}). We will show that $X$ belongs to the class 0-$\mathcal{D}$.  Note that $X$ is a countable union of arcs $J_n~(n\in \N)$ and we can choose such a retraction $r_{\epsilon}$ such that  $r_{\epsilon}|J_n$ is injective and hence 
 it is a doubly 0-dimensional map. By the same arguments as the proof of Proposition 6.7, we see that 
$X$ belongs to  0-$\mathcal{D}$. 
\end{proof}

\begin{cor} Let $X$ be one of the following spaces: PL-manifold, manifold with branch structures, Menger manifold, Sierpi\'nski carpet, Sierpi\'nski gasket and dendrite. Then  $X$ is a reconstruction space of dynamical systems. 
\end{cor}

\section{Application: Reconstructions of one-sided dynamical systems
from nonlinear time series analysis} 
There have been attempts to reconstruct dynamical models directly from data, and nonlinear methods for the analysis of time series data have been extensively investigated. This research is an inverse problem to the numerical analysis of dynamical systems model, in that it seeks to identify models that fit data.

Time-delay embedding is well-known for nonlinear time series analysis, and it is used
in several research fields such as physics, meteorology, informatics, neuroscience 
and so on. In laboratories, experimentalists are 
striving to find principles of phenomenons from a lot of data and they use delay embedding for reconstructing the dynamical systems from experimental time series. For smooth dynamical systems on manifolds, the celebrated Takens'
reconstruction theorem ensures validity of the delay embedding analysis. Takens' 
 theorem means that many dynamics theoretically can be  reconstructed  by the
delay coordinate system, more precisely almost all (two-sided) dynamical systems can be 
reconstructed from observation maps (see Takens \cite{Tak81,Tak02} and Sauer,Yorke and Casdagli \cite{SYC91}). So Takens'  theorem is the basis for nonlinear time series analysis and  form a bridge between the theory of nonlinear differential dynamical systems on smooth manifolds and nonlinear time series analysis. 

However, unfortunately the systems may not to be two-sided and moreover, they may not be  systems on manifolds. Recently we freqently encounter a situation where  we have to study dynamical systems of spaces that cannot have differential structure. In natural sciences and physical engineering, there has been an increase in importance of fractal sets and more complicated spaces, and also in mathematics, the dynamical properties and stochastic analysis of such spaces have been studied by many authors. Our reconstruction theorem theoretically ensures validity of the delay embedding analysis for (topological) dynamical systems on such complicated compact metric spaces,  
i.e., almost all one-sided dynamical systems $(X,T)$ of  spaces $X$ belonging to 0-$\mathcal{DCR}$ can be reconstructed from observation maps $f:X\to \R$ in the sense of  "trajectory embedding", i.e., the delay observation
map  $$I_{T,f}^{(0,1,2,\cdots,k)}: (X,T) \to (I_{T,f}^{(0,1,2,\cdots,k)}(X),\sigma_{T,f}^{(0,1,\cdots k)})$$
is a trajectory-embedding for a natural number $k \geq 2\dim X$, and so  the dynamical system
$$(I_{T,f}^{(0,1,2,\cdots,k)}(X),\sigma_{T,f}^{(0,1,2,\cdots,k)})$$ may reflect many dynamical properties of the original dynamical system $(X,T)$. Especially, $$I_{T,f}:[O(T)]\cong [O(\sigma_{T,f}^{(0,1,2,\cdots,k)})]~ \mbox{and}~ h(T)=h(\sigma_{T,f}^{(0,1,2,\cdots,k)}).$$  In laboratories, experimentalists may  
understand how the system $(X,T)$ will go in the future in the sense of orbital 
classification  from the analysis of experimental time series. More precisely, for $x,y \in X$, if one can find a time $n\in \N$ such that $$|\{i\in \N|~fT^i(x)=fT^i(y), 0\leq i \leq n \}|=2\dim X+1,$$ then $T^j(x)=T^j(y)$ for $j\geq n$ and hence $[o_T(x)]=[o_T(y)]$.

For more general case where a $d$-dimensional compact metric space $X$ does not belong to 0-$\mathcal{DCR}$ and $(X,T)$ is any one-sided dynamical system, we have an  extension $(\mu^d,T')$ of $(X,T)$, where $\mu^d$ is the $d$-dimensional Menger compactum containing $X$ and $T':\mu^d \to \mu^d$ is an extension of $T$  ~(see \cite{Bes88}). By Proposition 6.4, there is a possibility to be able to investigate the approximate properties of the dynamical system $(X,T)$ by use of time-delay embedding of the dynamical system $(\mu^d,T')$.  \\

{\bf Acknowledgments:}  This work was supported by JSPS KAKENHI
Grant Number JP19K03485.

\end{document}